\documentclass[12pt]{article}

\usepackage{amsmath}
\usepackage{amsthm} 
\usepackage{amsfonts}
\usepackage{latexsym}
\usepackage{amssymb}  
\usepackage{epsfig}

\usepackage{epsfig}

\usepackage{url}

\newcommand{\R}{\mathbb{R}}
\newcommand{\BR}{{[-\infty,\infty]}}

\newcommand{\calP}{\mathcal{P}}

\DeclareMathOperator{\dom}{dom}

\DeclareMathOperator{\range}{R}
\DeclareMathOperator{\graph}{Gr}

\newcommand{\inner}[2]{\langle{#1},{#2}\rangle}

\newcommand{\tos}{\rightrightarrows} 

\newtheorem{theorem}{Theorem}[section]
\newtheorem{lemma}[theorem]{Lemma}

\newtheorem{proposition}[theorem]{Proposition}

\newtheorem{definition}[theorem]{Definition}

\renewcommand{\hat}{\widehat}
\renewcommand{\tilde}{\widetilde}

\title{A maximal monotone operator of type (D) which 
maximal monotone extension to the bidual is not of type (D)}

\author{Orestes Bueno\footnote{Instituto de Mat\'ematica Pura e Aplicada 
(IMPA), Estrada Dona
Castorina 110, Rio de Janeiro, RJ, CEP 22460-320, Brazil,
{\tt obueno@impa.br}.
The work of this author author was partially supported by CNPq grant no.}
\qquad B. F. Svaiter\footnote{Instituto de Mat\'ematica Pura e Aplicada, 
(IMPA), Estrada Dona
Castorina 110, Rio de Janeiro, RJ, CEP 22460-320, Brazil,
{\tt benar@impa.br}.
The work of this author author was partially supported by
 CNPq grants no. 
474944/2010-7, 303583/2008-8 and  FAPERJ grant E-26/110.821/2008.
}}

\begin{document}

\maketitle

\begin{abstract}
We define a family of linear type (D) operators for which the inverse of their 
maximal monotone extensions to the bidual are not of type (D) and provide an
example of an operator in this family.  \\

keywords: maximal monotone, type (D), Banach space, extension, bidual. 
\end{abstract}

\section{Introduction}
Let $U$, $V$ arbitrary sets. A \emph{point-to-set} (or multivalued) operator
$T:U\tos V$ is a map $T:U\to \calP(V)$, where $\calP(V)$ is the power set
 of $V$. Given 
$T:U\tos V$, the \emph{graph} of $T$ is the set
\[
\graph(T):=\{(u,v)\in U\times V\:|\:v\in T(u)\},
\]
the  \emph{domain} and the \emph{range} of $T$ are, respectively,
\[
\dom(T):=\{u\in U\:|\:T(u)\neq\emptyset\},\qquad
\range(T):=\{v\in V\:|\:\exists u\in U, \;v\in T(u) \}
\]
and the \emph{inverse} of $T$ is the  point-to-set operator $T^{-1}:V\tos U$,
\[
 T^{-1}(v)=\{u\in U\:|\:v\in T(u)\}.
\]
A point-to-set operator $T:U\tos V$ is called \emph{point-to-point} if for
every $u\in\dom(T)$, $T(u)$ has only one element and in this case  we use the notation
$T:U\to V$ (which \emph{does not} mean  that $\dom(T)=U$).
Trivially, a point-to-point operator is injective if, and only if,
its inverse is also point-to-point.

Let $X$ be a real Banach space. We use the notation $X^*$ for the topological
dual of $X$. We will identify $X$ with its canonical injection into $X^{**}=(X^*)^*$
and we will use the notation $\inner{\cdot}{\cdot}$ for the duality product,
\[
\inner{x}{x^*}=\inner{x^*}{x}=x^*(x),\qquad x\in X,x^*\in X^*.
\]
A point-to-set operator $T:X\tos X^*$ (respectively $T:X^{**}\tos X^*$) is \emph{monotone}, if
\[
\inner{x-y}{x^*-y^*}\geq 0,\quad \forall (x,x^*),(y,y^*)\in \graph(T),
\]
(resp. $\inner{x^*-y^*}{x^{**}-y^{**}}\geq 0$, $\forall (x^{**},x^*),(y^{**},y^*)\in \graph(T)$), and it is \emph{maximal monotone} if it is monotone and maximal in the family of monotone operators in $X\times X^*$ (resp. $X^{**}\times X^*$)
with respect to the order of inclusion of the graphs.

Let $X$ be a \emph{non-reflexive} real Banach space and $T:X\tos X^*$ be
maximal monotone.
Since $X\subset X^{**}$, the point-to-set operator $T$ can
also be regarded as an operator from $X^{**}$ to $X^*$. We denote $\hat{T}:X^{**}\tos X^*$ as the operator such that
\[
\graph(\hat{T})=\graph(T).
\]
If $T:X\tos X^*$ is maximal monotone then $\hat{T}$ is (still) trivially monotone but, in general, not maximal monotone. 
%
%
Direct use of  the Zorn's Lemma shows that $\widehat{T}$ has a 
maximal monotone extension.
%
So it is natural to ask if such maximal monotone extension to the bidual is
unique. Gossez \cite{JPGos0,JPGos1,JPGos3,JPGos2} gave a sufficient condition for uniqueness of
such an extension.

\begin{definition}[\cite{JPGos0}]
  Gossez's \emph{monotone closure} (with respect to $X^{**}\times X^*$) 
of a maximal monotone operator $T:X\tos X^*$, is 
the point-to-set operator $\widetilde T:X^{**}\tos X^*$ whose graph
$\graph\left(\widetilde T\right)$ is
given by
\[
\graph\left(\widetilde T\right) = \{(x^{**},x^*)\in X^{**}\times X^*\:|\:\inner{x^*-y^*}{x^{**}-y}\geq 0,\,\forall (y, y^*) \in T\}. 
\]
A maximal monotone operator $T : X\tos X^* $, is of \emph{Gossez type (D)} if for any
$(x^{**},x^*) \in \graph\left( \widetilde T \right)$ , there exists a
bounded net $\bigg((x_i, x^*_i)\bigg)_{i\in }$ in $\graph(T)$
which converges to $(x^{**}, x^*)$ in the $\sigma(X^{**},
X^*)\times$strong topology of $X^{**}\times X^*$.
\end{definition}

Gossez proved~\cite{JPGos2} that a maximal monotone operator $T:X\tos X^*$ of
of type (D) has unique maximal monotone extension to the 
bidual, namely, its Gossez's monotone closure $\widetilde T:X^{**}\tos X^*$.
Beside this fact, maximal monotone operators of type (D) share 
many properties with maximal monotone operators defined in \emph{reflexive} Banach
spaces, as for example, convexity of the closure of the domain and convexity
of the closure of the range \cite{JPGos0}.

If $f:X\to\BR$ is a proper convex and lower semi-continuous function, then
$\partial f:X\tos X^*$
is of type (D) and the unique maximal monotone extension of $\partial f$ to
$X^{**}\times X^*$ is $(\partial f^*)^{-1}$, which is the inverse of a new type
(D) operator, namely $\partial f^*: X^*\tos X^{**}$.
It is natural to ask if this ``hereditary'' property of the subdifferential is
shared by other maximal monotone operators of type (D). Explicitly, if $T:X\tos X^*$ is of type (D),
is $(\widetilde{T}\;)^{-1}:X^*\tos X^{**}$ also of type (D)?  
The answer to this question is negative in general. We will provide an example
of a type (D) operator $T:X\tos X^*$ for which the inverse of its (unique)
maximal monotone extension to
$X^{**}\times X^*$ is \emph{not} of type (D).

We will need the following characterization of maximal monotone operators of type (D),
provided in \cite[eq.\ (5) and Theorem 4.4, item 2]{BSMMA-TypeD}.

\begin{theorem}
  \label{th:cr}
  A maximal monotone operator $T:X\tos X^*$ is of type (D) if and only if
  \[
  \sup_{(y,y^*)\in T}\inner{y}{x^*}+\inner{y^*}{x^{**}}-\inner{y^*}{y}\geq\inner{x^*}{x^{**}}
  \]
  for any $(x^*,x^{**})\in X^*\times X^{**}$.
\end{theorem}

\section{A special family of linear type (D) operators}
In this section we will characterize a family of maximal monotone operators of type (D)
for which the inverse of its maximal monotone extensions to the bidual are not of
type (D).

\begin{theorem}\label{thm:main}
Let $A:X^*\to X$ be linear, monotone, injective and everywhere defined,
and let
\[ 
 T=A^{-1}:R(A)\subset X\to X^*.
\]
 If there exists $x^{**}_0\in X^{**}\setminus X$ such that
\begin{equation}\label{eq:hipo}
  \sup_{x^*\in X^*} \inner{x^*}{x^{**}_0}-\inner{A(x^*)}{x^*}<\infty  
\end{equation}
then
\begin{enumerate}
  \item $T$ is maximal monotone of type (D).
  \item $\tilde{T}=\hat{T}$.
\item $(\,\widetilde{T}\;)^{-1}:X^*\tos X^{**}$ is not of type (D) on $X^*\times X^{**}$.
\end{enumerate}
\end{theorem}

\begin{proof}
  Since $A$ is injective, linear and monotone, the operator $T=A^{-1}:R(A)\to X^*$   is point-to-point, linear and  monotone.
Moreover
\[ \range (T)=\dom (A)=X^*.
\]
Therefore, using Theorem 6.7 of~\cite{PheSim}
we conclude that
\begin{description}
\item[a)] $T$ is maximal monotone of type (D) in $X\times X^*$,
\item[b)] $\hat{T}$ is maximal monotone in $X^{**}\times X^*$,
\end{description}
which proves item {\it 1} and {\it 2}.

To prove item {\it 3}, let
\[ 
\beta=  \sup_{x^*\in X^*}\inner{x^*}{x^{**}_0}-\inner{A(x^*)}{x^*}<\infty.
\]
As $X$ is a closed subspace of $X^{**}$,  and $x^{**}_0\in X^{**}\setminus X$, there exists
$x^{***}_0:X^{**}\to \R$, such that
\[ 
x^{***}_0\in X^{***},\qquad x^{***}_0(x)=0\quad \forall x\in X,
\qquad \inner{x^{**}_0}{x^{***}_0}> \beta.
\]
Since $\widetilde T=\widehat T$, using the definition of  $\widehat T$ we have 
\[
(\widetilde T\;)^{-1}(x^*)=(\widehat T\,)^{-1}(x^*)=T^{-1}(x^*)=A(x^*)\in X,
\]
for any $x^*\in X^*$.
Therefore $(\,\widehat T\,)^{-1}$ is point-to-point and
\begin{align*}
  \sup_{(x^*,x^{**})\in (\widetilde{T})^{-1}} \inner{x^*}{x^{**}_0}
    +\inner{x^{**}}{x^{***}_0}-\inner{x^*}{x^{**}}
  &=\sup_{x^*\in X^*} \inner{x^*}{x^{**}_0}+\inner{A(x^{*})}{x^{***}_0}-\inner{A(x^*)}{x^*}\\
  &=\sup_{x^*\in X^*} \inner{x^*}{x^{**}_0}-\inner{A(x^*)}{x^*}\\
  &=\beta<\inner{x^{**}_0}{x^{***}_0}.	
\end{align*}
To end the proof, use the above inequality and Theorem~\ref{th:cr} to conclude that
$(\,\widetilde T\,)^{-1}$ is not of type (D).
\end{proof}

\section{The family defined in Theorem~\ref{thm:main} is non-empty}

In this section  we will provide an explicit
example of an operator in the family described in Theorem~\ref{thm:main}.
Form now on, $\ell^1=\ell^1(\mathbb{N})$ and $\ell^\infty=\ell^\infty(\mathbb{N})$.
The family of elements of $\ell^\infty$ which converges to $0$ will be denoted by
$\ell^\infty_0$,
\begin{equation}
  \ell^\infty_0:=\left\{x\in \ell^\infty\;\left|\; 
  \lim_{i\to\infty} x_i=0
    \right.\right\}.
  \label{eq:linf0}
\end{equation}
We will use the canonical identifications
\[ (\ell^\infty_0)^*=\ell^1,\qquad (\ell^1)^*=\ell^\infty.
\]

Gossez defined in~\cite{JPGos1} the following operator
\begin{equation}
  \label{eq:gt}
G:\ell^1\to\ell^\infty,\qquad (G x)_{n}:=\sum_{k>n}x_k-\sum_{k<n} x_k,
\end{equation}
which is linear, continuous, anti-symmetric and maximal monotone.
Let 
\begin{equation}
  \label{eq:sig.e}
e:=(1,1,1,\dots)\in\ell^{\infty}.
\end{equation}

\begin{lemma}\label{lem:premain}
  For any $x\in \ell^1$, $G(x)+\inner{x}{e}e\in \ell^\infty_0$.
\end{lemma}
\begin{proof}
  Direct use of definitions \eqref{eq:gt} and \eqref{eq:sig.e}
  shows that for any $x\in \ell^1$,
  \[
  \lim_{n\to\infty} (G x)_n=-\inner{x}{e},
  \]
  which trivially implies that  $G(x)+\inner{x}{e}e\in
  \ell^\infty_0$. 
\end{proof}

\begin{proposition}
Let $X=\ell^{\infty}_0$. The operator
\begin{equation}
  \label{eq:op.a}
A:\ell^1\to\ell^\infty_0,\quad x\mapsto G(x)+\inner{x}{e}e
\end{equation}
satisfies the assumptions of Theorem~\ref{thm:main}. Hence $A^{-1}$ is a maximal
monotone operator of type (D) and the inverse of its unique maximal monotone
extension to the bidual is not of type (D).
\end{proposition}

\begin{proof}
By Lemma~\ref{lem:premain}, $A$ is everywhere defined. Linearity of $A$ follows from \eqref{eq:gt} and
\eqref{eq:sig.e}. Since Gossez's operator $G$ is anti-symmetric, for any $x\in\ell^1$
\begin{align}
  \label{eq:sp}
  \inner{A(x)}{x}=\inner{G(x)+\inner{x}{e}e}{x}=
  \inner{\inner{x}{e}e}{x}=\inner{e}{x}^2\geq 0,
\end{align}
which shows that $A$ is monotone.


To prove that $A$ is injective, it suffices to
show that its kernel is trivial. Suppose that 
\[
 x\in\ell^1, \qquad
A(x)=0.
\]
Using \eqref{eq:sp} we conclude that 
 $\inner{x}{e}^2=\inner{A(x)}{x}=0$.
 Therefore,  $\inner{e}{x}=0$, which combined with \eqref{eq:op.a} yields 
$G(x)=A(x)=0$. Using \eqref{eq:gt} we have
\[
 0=(Gx)_n-(Gx)_{n+1}=x_n+x_{n+1},\qquad n=1,2,\dots
\]
Hence $x_n=(-1)^{n+1}x_1$. Since $x\in \ell^1$, we conclude that
$x_n=0$ for all $n$.

We claim that 
\begin{equation}
  \label{eq:claim}
\sup_{x\in \ell^1} \inner{x}{e}-\inner{A(x)}{x}=\frac 14<\infty.
\end{equation}
Indeed, since $t-t^2\leq \frac{1}{4}$, for any $t\in\R$, using
\eqref{eq:sp} we have
\[
\inner{x}{e}-\inner{A(x)}{x}=\inner{x}{e}-\inner{x}{e}^2\leq \frac{1}{4},
\]
for any $x\in\ell^1$, with equality at $x=(1/2,0,0,\dots)$. 
To end the proof that $A$ satisfies the assumptions of Theorem~\ref{thm:main},
note that
\[
e\in \ell^\infty\setminus \ell^\infty_0=(\ell^\infty_0)^{**}\setminus\ell^\infty_0.
\]
which, combined with \eqref{eq:claim} shows that \eqref{eq:hipo} is satisfied
with $x^{**}_0=e$.

The second part of the propositions follows from the first part and Theorem~\ref{thm:main}.
\end{proof}

For $A$ as in \eqref{eq:op.a}, the operator $T:\ell^\infty_0\times\R\tos\ell^1\times\R$
whose graph is
\[
\graph(T)=\{ ( (A(x^*),t),(x^*,t^3))\;|\; x^*\in \ell^1,t\in\R\}
\]
is an example of a non-linear type (D) operator for which the inverse of
its maximal monotone to the bidual is not of type (D).


\end{document}